\documentclass{article}
\usepackage{maa-monthly,mathtools}
\usepackage{amsfonts,amsmath,amssymb,amsthm,tikz}
\usepackage{hyperref,url}
\usetikzlibrary{calc,intersections,through,backgrounds, arrows,decorations.markings,arrows.meta,math}
\usepackage{tkz-euclide}

\usepackage{graphicx} 
\usepackage{xcolor}
\usepackage[todonotes={textsize=footnotesize}]{changes}
\definechangesauthor[color=blue]{OK}
\definechangesauthor[color=red]{MA}

\title{Snell meets Fagnano. Path optimization through an imperfect mirror.}
\author{Maxim Arnold \and Jaewoo Park}
\date{}

\newtheorem{lemma}{Lemma}
\newtheorem{proposition}{Proposition}
\newtheorem{theorem}{Theorem}
\newtheorem{problem}{Problem}
\newtheorem*{problem*}{Problem}

\theoremstyle{definition}
\newtheorem{definition}{Definition}

\newtheorem{remark}{Remark}
\newtheorem{corollary}{Corollary}

\newcommand{\dis}{\mathsf{d}}

\newcommand{\car}{\mathcal{C}}
\def\s{s}
\def \a {\varphi}
\newcommand{\norm}[1]{\lvert #1 \rvert}

\begin{document}

\maketitle
\begin{abstract}
The renowned Fagnano problem asks for the inscribed triangle of minimal perimeter within a given reference triangle. Equivalently, it seeks a billiard trajectory inside the triangle that closes after three reflections. In this note, we consider a modification of this classical problem: finding the inscribed triangle of minimal \emph{weighted} perimeter -- or, equivalently, the periodic trajectory of a \emph{Snell billiard}.
\end{abstract}
\section{Introduction}

A classic puzzle, known since the Middle Ages, asks:

\begin{problem}
\label{prb:river} Given two villages $A$ and $B$ on the same bank of a river, find the shortest road connecting them so that both villages have access to the water.
\end{problem}

The solution is a broken line whose intermediate point lies at the intersection of the riverbank with the line segment joining one village to the reflection of the other across the river. (See the left panel of Fig.~\ref{fig: river}.)

A more elegant description is as follows: consider two copies of the Euclidean half-plane glued along their boundary (the riverbank). The shortest road corresponds to the \emph{geodesic} connecting point $A$ on one half-plane to point $B$ on the other.

Now, let us modify the problem. Suppose one village is much larger than the other, and we wish to minimize the total “effort” weighted by population. The two legs of the broken line should then be weighted by different coefficients $\lambda_1$ and $\lambda_2$, reflecting the populations. In geometric terms, we seek a geodesic on the same glued surface, but with the metrics $\dis(A,\cdot)$ and $\dis(B,\cdot)$ on the two half-planes scaled by $\lambda_1$ and $\lambda_2$.

\begin{lemma}\label{lm:Snell} Let $X$ be the point on the river bank, minimizing the weighted sum $$d(X):=\lambda_1\dis(A,X)+\lambda_2 \dis(B,X).$$ Then the angles between the normal to the river bank at $X$ and the rays $AX$ and $XB$ satisfy Snell's law of refraction (see the right panel of Fig. \ref{fig: river}):
\begin{equation}
	\label{eq: snells}
    \dfrac{\sin \alpha}{\sin \beta }=\dfrac{\lambda_1}{\lambda_2}=:\varkappa.
	\end{equation}

\end{lemma}

\begin{figure}[!hbt]
\centering
\begin{tikzpicture}
\begin{scope}[line cap=round,line join=round]
	\fill[color=gray!10](-2.5,0)--(2.5,0)--(2.5,-1.2)--(-2.5,-1.2)--(-2.5,0);
    \draw(-2.2,1.1)--(0,0);
	\draw[red](0,0)--(1.5,0.75);
        \draw[thick](-2.5,0)--(2.5,0);
	\draw[red,dashed](0,0)--(1.5,-0.75);
    \node at (-2.2,1.35) {$A$};
    \node at (1.69,0.75) {$B$};
    \node at (1.75,-0.75) {$B^*$};
    \node[circle,fill=black,inner sep=1pt] (B) at (1.5,0.75){};
     \node[circle,fill=black,inner sep=1pt] (A) at (-2.2,1.1){};
      \node[circle,fill=black,inner sep=1pt] (B1) at (1.5,-0.75){};
	\end{scope}
    
    \begin{scope}[xshift=6cm, line cap=round,line join=round]
	\fill[color=gray!30](-2.5,0)--(2.5,0)--(2.5,-1.2)--(-2.5,-1.2)--(-2.5,0);
	\fill[color=gray!20](0,0.5) arc (90:63:0.5)--(0,0)--cycle;
	\draw[red,->,>=stealth'](-1,1)--(0,0)--(0.5,1);
	\draw[dashed](0,0)--(0.5,-1);
	\draw[dotted](0,1)--(0,-1);
    \put(188,24) {\small$\beta$};
    \put(151,24) {\small$\alpha$};
	\end{scope}
\end{tikzpicture}
\caption{Left: River access optimiaztion problem. Right: Snell's law of refraction.} \label{fig: river}
\end{figure}
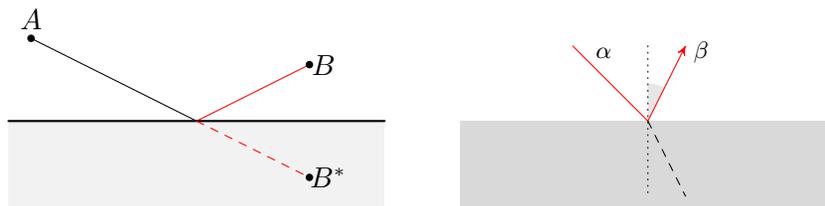

\begin{proof}
   Indeed, $X$ minimizes $d(X)$ if and only if the derivative of $d(X)$ in the direction of river bank is zero, i.e. the gradient of $d(X)$ at $X$ should be orthogonal to the river. 
The vector $\nabla d(X)$ is the sum $\lambda_1 v_A+\lambda_2 v_B$, where $v_A$ (resp. $v_B$) is the unit vector in the direction from $X$ to $A$ (resp. $B$). The horizontal components of these vectors cancel precisely when 
    $\lambda_2\sin(\alpha) = \lambda_1\sin(\beta)$.
        \end{proof}

        Hence, instead of finding the ray from $A$ that passes through $B$ after an ordinary reflection in the riverbank, one should look for the ray from $A$ that reaches $B$ after a \emph{Snell reflection} at the bank. Equivalently, while the classical problem seeks an ellipse with foci at $A$ and $B$ tangent to the river, the modified problem asks for a quartic curve -- known as the \emph{Cartesian oval} (see, e.g., \cite{Bacon}) -- with the same property.

In this note, we describe a similar modification of another famous path-minimization problem, which, remarkably, still admits an elementary geometric construction.

\subsection{Snell--Fagnano Problem.}
Consider an acute triangle $ABC$ and wrap a rubber band around its sides. What shape will we see? This question constitutes the famous \emph{Fagnano's problem}. 

 The rubber band would end up in a position that minimizes its total length. The classical resolution of this question states that the rubber band should connect the feet of the triangular altitudes $AA'$, $BB'$ and $CC'$, thus forming the \emph{orthic triangle} (see the left panel of the Fig. \ref{fig:Fagnano}) Another, perhaps fancier, way to say this is that the ``minimizer is the \emph{pedal triangle} of the orthocenter".

Among several proofs of this beautiful result we emphasise the \emph{unwrapping argument}: Let $A_1$ be the reflection of $A$ in the side $BC$, $B_1$ be the reflection of $B$ in the side $CA_1$, etc. The altitude $BB'$ under the first reflection is mapped to the altitude $BB_1'$ (see right panel of Fig. \ref{fig:Fagnano}) and so the points $C'$, $A'$ and the reflection of the point $B'$ are collinear. After six reflections, one obtains the triangle $A_2B_2C_2$ being the translation of the original triangle $ABC$. The reflections of the sides of the orthic triangle form a straight line, connecting a point on the side $BC$ with its image under the above translation. Hence such a triangle indeed has the minimal possible sum of the side lengths\footnote{ We leave the proof of uniqueness to the interested reader as a nice exercise in elementary geometry.}. 
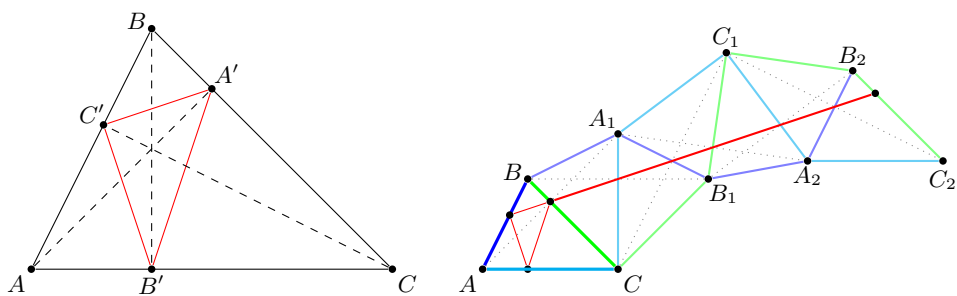
\begin{figure}[!hbt]
\begin{tikzpicture}
    \begin{scope}
    \tikzset{
			dot/.style={circle,inner sep=1pt,fill,name=#1}} 
		\tikzset{decoration={
			markings,
			mark=at position 0.5 with {\arrow{Stealth}}}}
           \tikzmath{\rad = 0.8; \oy=-1; \ox=0; 
\ax = 4; \ay=0; \bx=0; \by = 4; \cx=-2;\cy=0; 
\mx=\ax*(\by*\by+\ax*\cx)/(\ax*\ax+\by*\by);
\my=(\by/\ax)*(\ax-\mx);
\nx=\cx*(\by*\by+\ax*\cx)/(\cx*\cx+\by*\by);
\ny=(\by/\cx)*(\cx-\nx);
} 
\node[dot=$O$] (O) at (\ox,\oy){};
    \node[dot=$A$] (A) at (\ax*\rad+\ox,\ay*\rad+\oy){};
    \node[dot=$B$] (B) at (\bx*\rad+\ox,\by*\rad+\oy){};
    \node[dot=$C$] (C) at (\cx*\rad+\ox,\cy*\rad+\oy){};
    \draw(A)--(B)--(C)--(A);
    \node[dot=$M$] (M) at (\mx*\rad+\ox,\my*\rad+\oy){};
   \node[dot=$N$] (N) at (\nx*\rad+\ox,\ny*\rad+\oy){};
   
        \draw[red](O)--(M)--(N)--(O);
        \draw[dashed] (B)--(O);
        \draw[dashed] (A)--(N);
        \draw[dashed] (C)--(M);

        \node at (\cx*\rad+\ox-0.2,\cy*\rad+\oy-0.2) {\small$A$};
        \node at (\ax*\rad+\ox+0.2,\ay*\rad+\oy-0.2) {\small$C$};
        \node at (\bx*\rad+\ox-0.2,\by*\rad+\oy+0.1) {\small$B$};
\node at (\ox,\oy-0.2) {\small$B'$};
  \node at (\mx*\rad+\ox+0.2,\my*\rad+\oy+0.2){\small$A'$};
     \node at (\nx*\rad+\ox-0.15,\ny*\rad+\oy+0.15){\small$C'$};

    \end{scope}
\begin{scope}[xshift=7cm]
    \tikzset{
			dot/.style={circle,inner sep=1pt,fill,name=#1}} 
		\tikzset{decoration={
			markings,
			mark=at position 0.5 with {\arrow{Stealth}}}}
           \tikzmath{\rad = 0.3; \oy=-1; \ox=-2; 
\ax = 4; \ay=0; \bx=0; \by = 4; \cx=-2;\cy=0; 
\mx=\ax*(\by*\by+\ax*\cx)/(\ax*\ax+\by*\by);
\my=(\by/\ax)*(\ax-\mx);
\nx=\cx*(\by*\by+\ax*\cx)/(\cx*\cx+\by*\by);
\ny=(\by/\cx)*(\cx-\nx);
\cpx=\cx-2*(\by-\ay)*((\cx-\ax)*(\by-\ay)+(\cy-\ay)*(\ax-\bx))/((\ax-\bx)^2+(\ay-\by)^2);
\cpy=\cy-2*(\ax-\bx)*((\cx-\ax)*(\by-\ay)+(\cy-\ay)*(\ax-\bx))/((\ax-\bx)^2+(\ay-\by)^2);
\bpx=\bx-2*(\cpy-\ay)*((\bx-\ax)*(\cpy-\ay)+(\by-\ay)*(\ax-\cpx))/((\ax-\cpx)^2+(\ay-\cpy)^2);
\bpy=\by-2*(\ax-\cpx)*((\bx-\ax)*(\cpy-\ay)+(\by-\ay)*(\ax-\cpx))/((\ax-\cpx)^2+(\ay-\cpy)^2);
\apx=\ax-2*(\cpy-\bpy)*((\ax-\bpx)*(\cpy-\bpy)+(\ay-\bpy)*(\bpx-\cpx))/((\bpx-\cpx)^2+(\bpy-\cpy)^2);
\apy=\ay-2*(\bpx-\cpx)*((\ax-\bpx)*(\cpy-\bpy)+(\ay-\bpy)*(\bpx-\cpx))/((\bpx-\cpx)^2+(\bpy-\cpy)^2);
\cppx=\cpx-2*(\bpy-\apy)*((\cpx-\apx)*(\bpy-\apy)+(\cpy-\apy)*(\apx-\bpx))/((\apx-\bpx)^2+(\apy-\bpy)^2);
\cppy=\cpy-2*(\apx-\bpx)*((\cpx-\apx)*(\bpy-\apy)+(\cpy-\apy)*(\apx-\bpx))/((\apx-\bpx)^2+(\apy-\bpy)^2);
\bppx=\bpx-2*(\cppy-\apy)*((\bpx-\apx)*(\cppy-\apy)+(\bpy-\apy)*(\apx-\cppx))/((\apx-\cppx)^2+(\apy-\cppy)^2);
\bppy=\bpy-2*(\apx-\cppx)*((\bpx-\apx)*(\cppy-\apy)+(\bpy-\apy)*(\apx-\cppx))/((\apx-\cppx)^2+(\apy-\cppy)^2);
\appx=\apx-2*(\cppy-\bppy)*((\apx-\bppx)*(\cppy-\bppy)+(\apy-\bppy)*(\bppx-\cppx))/((\bppx-\cppx)^2+(\bppy-\cppy)^2);
\appy=\apy-2*(\bppx-\cppx)*((\apx-\bppx)*(\cppy-\bppy)+(\apy-\bppy)*(\bppx-\cppx))/((\bppx-\cppx)^2+(\bppy-\cppy)^2);
\npx=0-2*(\by-\ay)*((0-\ax)*(\by-\ay)+(0-\ay)*(\ax-\bx))/((\ax-\bx)^2+(\ay-\by)^2);
\npy=0-2*(\ax-\bx)*((0-\ax)*(\by-\ay)+(0-\ay)*(\ax-\bx))/((\ax-\bx)^2+(\ay-\by)^2);
\nppx=\mx-2*(\cpy-\ay)*((\mx-\ax)*(\cpy-\ay)+(\my-\ay)*(\ax-\cpx))/((\ax-\cpx)^2+(\ay-\cpy)^2);
\nppy=\my-2*(\ax-\cpx)*((\mx-\ax)*(\cpy-\ay)+(\my-\ay)*(\ax-\cpx))/((\ax-\cpx)^2+(\ay-\cpy)^2);
\nppx=\mx+4.8*(\npx-\mx);
\nppy=\my+4.8*(\npy-\my);
}

\node[dot=$O$] (O) at (\ox,\oy){};
    \node[dot=$A$] (A) at (\ax*\rad+\ox,\ay*\rad+\oy){};
    \node[dot=$B$] (B) at (\bx*\rad+\ox,\by*\rad+\oy){};
    \node[dot=$C$] (C) at (\cx*\rad+\ox,\cy*\rad+\oy){};
    \draw[very thick, color=green] (A)--(B);
    \draw[very thick,color=blue] (B)--(C);
    \draw[very thick, color=cyan] (A)--(C);

\node[dot=$C1$] (C1) at (\cpx*\rad+\ox,\cpy*\rad+\oy){};
\node[dot=$B1$] (B1) at (\bpx*\rad+\ox,\bpy*\rad+\oy){};
\node[dot=$A1$] (A1) at (\apx*\rad+\ox,\apy*\rad+\oy){};
\node[dot=$C2$] (C2) at (\cppx*\rad+\ox,\cppy*\rad+\oy){};
\node[dot=$B2$] (B2) at (\bppx*\rad+\ox,\bppy*\rad+\oy){};
\node[dot=$A2$] (A2) at (\appx*\rad+\ox,\appy*\rad+\oy){};

    \draw[thick,color=blue!50] (B)--(C1);
    \draw[thick, color=cyan!50] (A)--(C1);
    
\draw[thick, color=green!50] (A)--(B1);
    \draw[thick,color=blue!50] (B1)--(C1);

\draw[thick, color=green!50] (A1)--(B1);
 \draw[thick, color=cyan!50] (A1)--(C1);
    
    \draw[thick,color=blue!50] (B1)--(C2);
    \draw[thick, color=cyan!50] (A1)--(C2);

\draw[thick, color=green!50] (A1)--(B2);
    \draw[thick,color=blue!50] (B2)--(C2);

\draw[thick,color=green!50] (A2)--(B2);
    \draw[thick, color=cyan!50] (A2)--(C2);  
    
    \draw[gray,dotted] (B)--(B1);
    \draw[gray,dotted] (C)--(C1);
    \draw[gray,dotted] (A)--(A1);
     \draw[gray,dotted] (C1)--(C2);
     \draw[gray,dotted] (B1)--(B2);
     \draw[gray,dotted] (A1)--(A2);

    \node[dot=$M$] (M) at (\mx*\rad+\ox,\my*\rad+\oy){};
   \node[dot=$N$] (N) at (\nx*\rad+\ox,\ny*\rad+\oy){};
   \node[dot=$Np$] (Np) at (\nppx*\rad+\ox,\nppy*\rad+\oy){};
        \draw[red](O)--(M)--(N)--(O);
        \draw[thick, red](M)--(Np);

        \node at (\bx*\rad+\ox-0.2,\by*\rad+\oy+0.1) {\small$B$};
         \node at (\cx*\rad+\ox-0.2,\cy*\rad+\oy-0.2) {\small$A$};
        \node at (\ax*\rad+\ox+0.2,\ay*\rad+\oy-0.2) {\small$C$};

        \node at (\cpx*\rad+\ox-0.2,\cpy*\rad+\oy+0.2){\small$A_1$};
        \node at (\bpx*\rad+\ox+0.2,\bpy*\rad+\oy-0.2){\small$B_1$};
\node at (\apx*\rad+\ox,\apy*\rad+\oy+0.2){\small$C_1$};
\node at (\cppx*\rad+\ox,\cppy*\rad+\oy-0.2){\small$A_2$};
\node at (\bppx*\rad+\ox,\bppy*\rad+\oy+0.2){\small$B_2$};
\node at (\appx*\rad+\ox,\appy*\rad+\oy-0.25){\small$C_2$};

    \end{scope}
\end{tikzpicture}
    \caption{Left: Fagnano orbit. Right: Unwrapping argument.} 
	\label{fig:Fagnano}
\end{figure}

The argument above presents a dynamical interpretation of the Fagnano result: the orthic triangle is the only $3$-periodic billiard trajectory inside an acute triangle. That is, the light ray, reflected in the sides of the acute triangle $ABC$ closes up after three reflections if and only if it follows the sides of orthic triangle. Interestingly, this argument fails when $ABC$ is not acute. If $ABC$ is obtuse, the orthocenter lies on the exterior of $ABC$, and thus the pedal triangle does not belong to the interior of $ABC$. The minimizer degenerates to the double cover of the altitude of minimal length, and the unwrapped picture does not contain straight segments.

The question of whether periodic billiard trajectories exist in any obtuse triangle remains open and has garnered significant interest over the past decades (see, e.g., \cite{Schwartz2009} for an overview of recent progress).

In the present note we address the following modification of the Fagnano question.
\begin{problem}\label{prb:Fagngno}
Describe the length-minimizing configuration if the sides of the rubber band have different stiffness $(\lambda_A,\lambda_B,\lambda_C)$. 
\end{problem}

 Thanks to Lemma \ref{lm:Snell}, this problem can be reformulated with the dynamical systems flavor of finding the Snell's billiard trajectory with the refraction coefficients $(\varkappa_a,\varkappa_b,\varkappa_c)$ given by
\begin{equation}\label{eq:lambdakappa}\lambda_A:\lambda_B=\varkappa_c,\qquad \lambda_B:\lambda_C=\varkappa_a,\qquad \lambda_C:\lambda_A=\varkappa_b.\end{equation}
Here refraction coefficients are associated with the sides of the triangle $ABC$.
Multiplying the above identities, we get  $\varkappa_a\varkappa_b\varkappa_c = 1$ so that after three reflections we return back to the usual metric. Our main results can be summarized as follows:
\begin{itemize}
\item Every $3$-periodic trajectory of a Snell billiard inside a triangle corresponds to the pedal triangle of a certain point $F$, which we call the \emph{Snell–Fagnano point} (see Lemma \ref{lm:point}).

\item The existence of a Snell–Fagnano point inside a triangle with side lengths $(a,b,c)$ and stiffness parameters $(\lambda_A,\lambda_B,\lambda_C)$ is equivalent to the existence of a point whose distances to the vertices are proportional to $(\lambda_A,\lambda_B,\lambda_C)$ (see Lemma \ref{lm:isogonalFagnano}).

\item The existence of the Snell–Fagnano point is further equivalent to the existence of a triangle with side lengths $(\lambda_A a,\, \lambda_B b,\, \lambda_C c)$ (see Lemma \ref{lm:Akopyan}).

\item Finally, we provide an elementary geometric construction of the Snell–Fagnano point and derive explicit conditions for it to lie inside the triangle $ABC$ (see Theorem \ref{th:main} and Corollary \ref{cor:main}). When these conditions fail, the corresponding Snell billiard trajectory degenerates into a double cover of the shortest altitude of $ABC$.

\end{itemize}

The rest of the note is organized as follows. Section \ref{sec:nomenclature} establishes some necessary notions from the triangle geometry and provide necessary and sufficient conditions for the existence of the Snell-Fagnano point. Section \ref{sec:construction} presents the main construction and discuss its properties. Section \ref{sec:other} addresses further possible generalizations of our result.

\section{Existence of the Snell-Fagnano point.}\label{sec:nomenclature}
 From now on, we denote the reference triangle by $ABC$, its angles by $\alpha$, $\beta$, $\gamma$, and its side lengths by $a$, $b$, $c$, respectively. Whenever no confusion arises, we will use the same letters to denote the entire lines containing the corresponding sides. Finally, $[ABC]$ will stand for the (signed) area of $ABC$ in the natural orientation.
 
 The following lemma asserts that every 3-periodic Snell billiard trajectory can be uniquely represented by choosing some specific point in the interior of the triangle and taking its pedal triangle.

\begin{lemma} \label{lm:point}
Let triangle $A'B'C'$ be inscribed in triangle $ABC$, with $A' \in a, B'\in b, C' \in c$.
Define $h_a$ as the normal to $a$ passing through $A'$, and define $h_b, h_c$ similarly.
Let $\eta_a$ denote the ratio between the sines of the angles formed by $h_a$ with the sides of the triangle $A'B'C'$ (see Fig. \ref{fig:FagnanoSnell}) i.e., 
$
\eta_a = \frac{\sin{\alpha_1}}{\sin{\alpha_2}}$, $
\eta_b = \frac{\sin{\beta_1}}{\sin{\beta_2}}$, 
$\eta_c = \frac{\sin{\gamma_1}}{\sin{\gamma_2}}
$. 
Then the lines $h_a$, $h_b$ and $h_c$ are concurrent if and only if
$\eta_a\eta_b\eta_c=1.$
\end{lemma}

\begin{proof}
    This follows from Ceva's theorem applied to the triangle $A'B'C'$. Denote by $\hat{B}$ the intersection of the line $h_b$ with the side $C'A'$ and define $\hat{A}$, $\hat{C}$ correspondingly.  Using area,
    \[
    \eta_a = \frac{\sin \alpha_1}{\sin \alpha_2} = \frac{\dis(C',\hat{B})}{\dis(A',\hat{B})} \cdot \frac{\dis(A',B')}{\dis(C',B')},
    \] and when the etas are multiplied, the $\frac{A'B'}{C'B'}$ factors cancel. By Ceva's theorem, the lines $h_a, h_b, h_c$ are concurrent if and only if 
    $
    \frac{\dis(A',\hat{B})}{\dis(C'\hat{B})}\frac{\dis(C'\hat{A})}{\dis(B'\hat{A})}\frac{\dis(B'\hat{C})}{\dis(A'\hat{C})}= 1
    $.
\end{proof}

Thanks to Lemma \ref{lm:point}, the normals to the sides of the reference triangle at the vertices of the minimal configuration are concurrent at some point. We will call this special point the \emph{Snell-Fagnano point} and denote it by $F_{\lambda}$, with $\lambda$ representing the vector $(\lambda_A,\lambda_B,\lambda_C)$. Thus, every length-minimizing configuration is the pedal triangle of some point inside the triangle and so, in order to construct the minimizer, we should address the question on how to construct the point $F_\lambda$.

\begin{lemma}\label{lm:SFpoint}
The Snell-Fagnano point $F_\lambda$ is the point satisfying the conditions 
\[
\frac{\sin\angle FCA}{\sin\angle FBA}=\varkappa_a, \:
\frac{\sin\angle FAB}{\sin\angle FCB}=\varkappa_b, \:
\frac{\sin\angle FBC}{\sin\angle FAC}=\varkappa_c
\]
with the refraction coefficients provided by the relations \eqref{eq:lambdakappa} (see Fig. \ref{fig:FagnanoSnell}). 
\end{lemma}
\begin{proof}
  Indeed, let $A'B'C'$ be the pedal triangle of the point $F$ with respect to triangle $ABC$. Then since the quadrilateral $AC'PB'$ is cyclic (i.e. can be inscribed in a circle), one gets $\angle FC'B'=\angle FAB'$.
    Thanks to the symmetry $A\leftrightarrow B\leftrightarrow C$, one gets five more identifications.
\end{proof}

\begin{figure}[!hbt]
    \centering
    \includegraphics[width=0.5\linewidth]{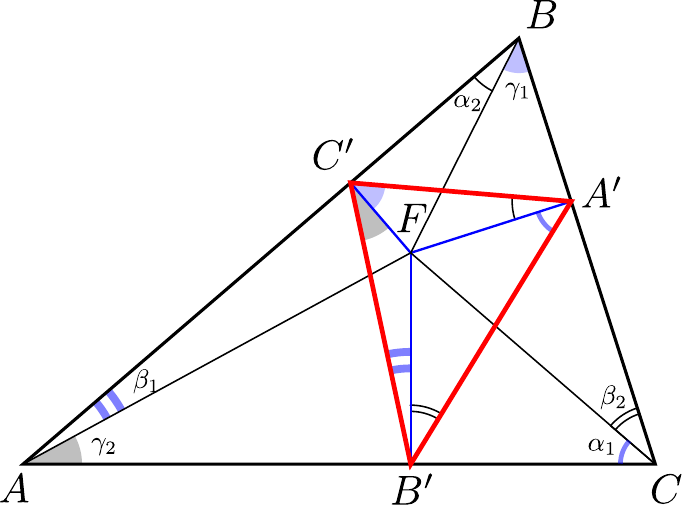}
    \caption{Snell-Fagnano point. }
    \label{fig:FagnanoSnell}
\end{figure}

\begin{remark}
    For any acute triangle $ABC$ and any point $F$ inside it there exist refraction coefficients $\varkappa_a, \varkappa_b, \varkappa_c$ such that the pedal triangle of $F$ is the Snell-Fagnano orbit for those coefficients (see e.g.\cite{Fagnano}).
\end{remark}

\begin{remark}
    Note that we have chosen a particular interpretation of the problem \ref{prb:Fagngno} as a \emph{counter clockwise} Snell billiard trajectory.  Clockwise trajectory can be represented as the pedal triangle of the point $F$ with the refraction coefficients $(\varkappa_a^{-1}, \varkappa_b^{-1}, \varkappa_c^{-1})$ and so is the path minimizing the weighted perimeter with coefficients $(\lambda_A^{-1},\lambda_B^1,\lambda_C^{-1})$.  
\end{remark}

Given the reference triangle $ABC$ one could assign several different parametrizations to the points in the plane. For the sake of clarity, we recall few of such parametrizations. 
\begin{definition}\label{def: coord}
For a point $P$ relative to $ABC$,
\begin{itemize}
    \item The \textit{barycentric coordinates} are the homogeneous triple $(\rho_a : \rho_b : \rho_c),$ where $\rho_a,\rho_b,\rho_c$ are proportional to the signed areas $[PBC]$, $[PCA]$ and $[PAB]$, respectively.  
    If $\rho_a+\rho_b+\rho_c=1$, then we say the coordinates are \textit{normalized} and for the radius vectors of the points the property
    ${\vec{P} = \rho_a \vec{A} + \rho_b \vec{B} + \rho_c \vec{C}}$ holds. 
     If normalized barycentric coordinates are all positive, the point $P$ is called \emph{convex combination} of the vertices of the triangle and thus belongs to the interior of $ABC$ (a.k.a. \emph{convex hull of the vertices}). 
     \item  The \textit{trilinear coordinates} are the homogeneous triple $(l_a:l_b:l_c)$, where $l_a$, $l_b$,  $l_c$ are proportional to the signed distances from $P$ to the sides $BC$, $CA$ and $AB$ respectively. Trilinear coordinates are related to the barycentric coordinates as $(\rho_a:\rho_b:\rho_c)=(a l_a:b l_b:c l_c)$.
    \item The \textit{tripolar coordinates} are the homogeneous triple $(r_A :r_B :r_C)$, where $r_A$ is the distance between $A$ and $P$. It should be noted that tripolar coordinates do not uniquely determine a point: a given triple may correspond to zero, one, or two distinct points in the plane. 
\end{itemize}
\end{definition}

\begin{definition}
    The \textit{isogonal conjugate} $P^*$ of a point $P$ with respect to a triangle $ABC$ is defined as the (unique) intersection of the reflection of the lines $PA, PB, PC$ about the angle bisectors of $\angle A, \angle B,\angle C$ respectively. 
\end{definition}

 \noindent The isogonal conjugacy map is a bijection on the complement to the lines through the vertices of the triangle, $\mathbb{R}^2\backslash\{AB \cup BC \cup CA\}$. Moreover, isogonal conjugation maps points inside the triangle to interior points and exterior points to exterior points. In trilinear coordinates, the isogonal conjugacy takes the form $$(l_a:l_b:l_c)\mapsto (l_a^{-1}:l_b^{-1}:l_c^{-1}).$$ 

\begin{lemma}\label{lm:isogonalFagnano}
    The tripolar coordinates of the isogonal conjugate of the Snell-Fagnano point of the triangle $ABC$ are 
     \begin{equation}\label{eq:SF*} F_\lambda^*=(\lambda_A:\lambda_B:\lambda_C).\end{equation}
\end{lemma}

\begin{proof}
Isogonal conjugation transforms $\angle FBA$ into $\angle F^*BC$, etc. so that 
\[
\frac{\sin\angle F^*CB}{\sin\angle F^*BC}=\varkappa_a, \:
\frac{\sin\angle F^*AC}{\sin\angle F^*CA}=\varkappa_b, \:
\frac{\sin\angle F^*BA}{\sin\angle F^*AB}=\varkappa_c.
\]
    From the sine law for the triangle $AF^*C$ we get $\frac{\dis(A,F^*)}{\dis(C,F^*)}=\frac{\sin\beta_2}{\sin \beta_1}=\varkappa_b$.  Similarly, we get
$
\frac{\dis(B,F^*)}{\dis(C,F^*)} = \varkappa_a$ and  
$\frac{\dis(A,F^*)}{\dis(B,F^*)} = \varkappa_c$ and the claim follows from the relations \eqref{eq:lambdakappa}. 
\end{proof}

\begin{definition}
    The \emph{Apollonian circle} $\odot_{A,B}(r)$ is the locus of points $P$ with
    $
    \frac{\dis(P,A)}{\dis(P,B)}~=~r.
    $    This locus is a circle, and the internal and external division points of $AB$ in the ratio $AP:PB = r:1$, denoted $M$ and $N$ respectively, is a diameter of this circle. (See Fig.~\ref{fig:Apollonian}).
\end{definition}

\begin{figure}[!hbt]
\centering
\begin{tikzpicture}[scale=0.8]
    \tikzset{
			dot/.style={circle,inner sep=1pt,fill,name=#1}} 
		\tikzset{decoration={
			markings,
			mark=at position 0.5 with {\arrow{Stealth}}}}
           \tikzmath{\rad = 0.8; \oy=0; \ox=0; 
} 

\draw[red] (\ox,\oy) circle(3*\rad);
\node[dot=$M$] (M) at (\ox-3*\rad,\oy){};
\node[dot=$N$] (N) at (\ox+3*\rad,\oy){};
\node[dot=$P$] (P) at (\ox+3*0.6*\rad,\oy+3*0.8*\rad){};
\node[dot=$A$] (A) at (\ox-6*\rad,\oy){};
\node[dot=$B$] (B) at (\ox-1*\rad,\oy){};
\draw[thick] (A)--(B);
\draw (M)--(N);
\draw[thin] (A)--(P)--(B);
\draw[dashed] (M)--(P)--(N);
\node at (\ox-6*\rad-0.2,\oy-0.2) {\small$A$};
\node at (\ox-1*\rad-0.1,\oy-0.2) {\small$B$};
\node at (\ox-3*\rad-0.2,\oy-0.2) {\small$M$};
\node at (\ox+3*\rad+0.2,\oy-0.2) {\small$N$};
      \node at (\ox+3*0.6*\rad+0.2,\oy+3*0.8*\rad+0.2){\small$P$};       
     \end{tikzpicture}   
      \includegraphics[width=0.43\linewidth]{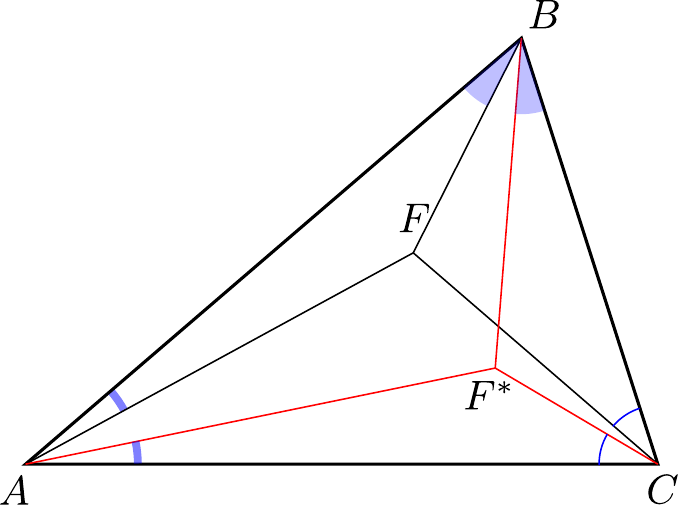}
    \caption{Left: Apollonian Circle $\odot_{A,B}(\lambda)$. Right: $F^*$ is the isogonal conjugate of $F$ with respect to triangle $ABC$}
    \label{fig:Apollonian}
\end{figure}
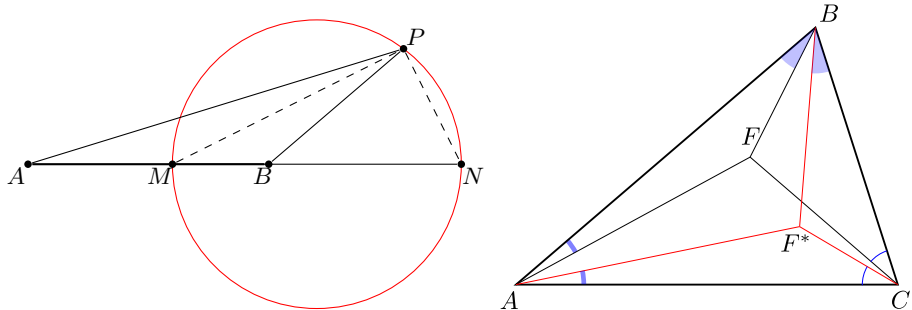

Hence, the existence of the point with the prescribed tripolar coordinates $(x~:~y~:~z)$ is equivalent to the concurrence of the three Apollonnian circles  $\odot_{A,B}(x/y)$, $\odot_{B,C}(y/z)$ and $\odot_{C,A}(z/x)$.

\begin{remark}
    In the case of $\lambda_A=\lambda_B=\lambda_C$, Lemma \ref{lm:isogonalFagnano} asserts that the isogonal conjugate of the orthocenter is the circumcenter and the corresponding Apollonian circles become perpendicular bisectors. Another special choice of the tripolar coordinates $(1/a:1/b:1/c)$  provide isodynamic points with isogonal conjugate Fermat-Torricelli points.
\end{remark}

\begin{lemma}\label{lm:Akopyan}
   If two Apollonian circles $\odot_{A,B}(\lambda_A/\lambda_B)$ and $\odot_{B,C}(\lambda_B/\lambda_C)$ have common point $P$ then $P$ belongs to the third Apollonian circle $\odot_{C,A}(\lambda_C/\lambda_A)$ and there exists a triangle $\tilde{\triangle}$ with the sides $(\lambda_A a,\lambda_B b, \lambda_C c)$.
\end{lemma}

\begin{proof}
Points on the circle $\odot_{A,B}(\lambda_A/\lambda_B)$ have tripolar coordinates $(\lambda_A t:\lambda_B t: z)$. If some of them $P$ belongs to the second circle, it follows that $\dis(P,A)=\lambda_A t$, $\dis(P,B)=\lambda_B t$, $\dis(P,C)=\lambda_C t$. The Ptolemey inequality for the quadrilateral $ABPC$ then reads 
$b \lambda_B t+c \lambda_C t\geqslant a \lambda_A t$. Which finishes the proof of the Lemma.   

Here we reproduce another,  elegant proof of the same result, provided by A. Akopyan \cite{Akopyan}. For the triangle $ABC$, consider the transformation (scaling and rotation), mapping $AC$ to $AB$. Denote by $C'$ and by $P'$ the images of the points $B$ and $P$ under this transformation (see Fig.\ref{fig:Akopyan}).
Therefore we have  $b\dis(P,B) = c \dis(P', C') $.
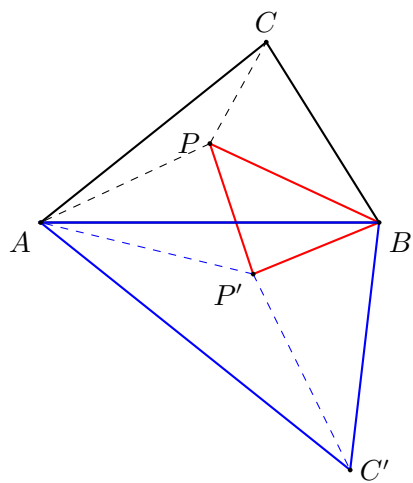
\begin{figure}[!hbt]
    \centering
    \begin{tikzpicture}[scale=3]

\pgfmathsetmacro{\Ax}{0}   
\pgfmathsetmacro{\Ay}{0}
\pgfmathsetmacro{\Bx}{1.5}   
\pgfmathsetmacro{\By}{0}
\pgfmathsetmacro{\Cx}{1.0} 
\pgfmathsetmacro{\Cy}{0.8}
\pgfmathsetmacro{\Px}{0.75}
\pgfmathsetmacro{\Py}{0.35}

\coordinate (A) at (\Ax,\Ay);
\coordinate (B) at (\Bx,\By);
\coordinate (C) at (\Cx,\Cy);
\coordinate (P) at (\Px,\Py);

\pgfmathsetmacro{\ABlen}{veclen(\Bx-\Ax,\By-\Ay)}
\pgfmathsetmacro{\AClen}{veclen(\Cx-\Ax,\Cy-\Ay)}
\pgfmathsetmacro{\s}{\ABlen/\AClen}

\pgfmathsetmacro{\angAB}{atan2(\By-\Ay,\Bx-\Ax)} 
\pgfmathsetmacro{\angAC}{atan2(\Cy-\Ay,\Cx-\Ax)}
\pgfmathsetmacro{\theta}{\angAB-\angAC}
\pgfmathsetmacro{\c}{cos(\theta)}
\pgfmathsetmacro{\si}{sin(\theta)}

\pgfmathsetmacro{\BX}{\Bx-\Ax} \pgfmathsetmacro{\BY}{\By-\Ay}
\pgfmathsetmacro{\Cpx}{\Ax + \s*(\c*\BX - \si*\BY)}
\pgfmathsetmacro{\Cpy}{\Ay + \s*(\si*\BX + \c*\BY)}
\coordinate (Cp) at (\Cpx,\Cpy); 

\pgfmathsetmacro{\PX}{\Px-\Ax} \pgfmathsetmacro{\PY}{\Py-\Ay}
\pgfmathsetmacro{\Ppx}{\Ax + \s*(\c*\PX - \si*\PY)}
\pgfmathsetmacro{\Ppy}{\Ay + \s*(\si*\PX + \c*\PY)}
\coordinate (Pp) at (\Ppx,\Ppy); 

\draw[red,thick] (Pp)--(B)--(P)--cycle;
\draw[thick] (A)--(B)--(C)--cycle;
\draw[thick,blue] (A)--(B)--(Cp)--cycle; 

\draw[dashed] (P)--(A) (P)--(C);
\draw[dashed,blue] (Pp)--(A) (Pp)--(Cp);

\fill (A) circle(0.3pt) node[below left] {$A$};
\fill (B) circle(0.3pt) node[below right] {$B$};
\fill (C) circle(0.3pt) node[above] {$C$};
\fill (Cp) circle(0.3pt) node[right] {$C'$};
\fill (P) circle(0.3pt) node[left] {$P$};
\fill (Pp) circle(0.3pt) node[below left] {$P'$};

    \end{tikzpicture}
    \caption{Proof of the Lemma \ref{lm:Akopyan}.}
    \label{fig:Akopyan}
\end{figure}

The triangle $APP'$ is similar to the triangle $ACB$ with coefficient of propotionality $\frac{\dis(A,P)}{\dis(A,C)}$. Hence $\dis(P,P')=\frac{a}{b}\dis(A,P)$. 
Hence for the triangle $PBP'$ one has
$$\dis(P',P):\dis(P,B):\dis(P',B)=a\dis(P,A):b\dis(P,B):c\dis(P,C).$$  

\end{proof}

\begin{remark}
    Any Apollonian circle cuts the circumcircle orthogonally and thus is mapped to itself under the inversion in the circumcircle. Hence, if the triangle with the sides $(\lambda_A a, \lambda_B b \lambda_c)$ exists, then the Apollonian circles intersect in the pair of inverse points, if such triangle is degenerate, then the intersection point is unique and belongs to the circumcircle. 
\end{remark}

\section{Main construction.}\label{sec:construction}
From Lemma \ref{lm:point} it follows that the Snell--Fagnano configuration with the weights $(\lambda_A,\lambda_B,\lambda_C)$ is uniquely described by the Snell-Fagnano point $F_\lambda$. Lemma \ref{lm:isogonalFagnano} states that the Snell-Fagnano point is isogonal conjugate of the point $F^*_\lambda$ with tripolar coordinates provided by the wieghts $(\lambda_A:\lambda_B:\lambda_C)$. 

Finally, Lemma~\ref{lm:Akopyan} states that the existence of the point $F^*_\lambda$ implies the existence of a triangle $\tilde{\triangle}$ with side lengths $(a\lambda_A,, b\lambda_B,, c\lambda_C)$. We now present a simple geometric construction of the point $F_\lambda$ and discuss the conditions under which it lies inside the triangle $ABC$. Our construction shows that the existence of the two objects mentioned in Lemma~\ref{lm:Akopyan} is, in fact, equivalent.

\begin{lemma}\label{lm:construction}
    Let $A_1$, $B_1$ and $C_1$ are taken in such way that the triangles $BA_1C$ $B_1AC$ and $BAC_1$ are similar to the triangle $\tilde{\triangle}$. Then the lines $AA_1$ $BB_1$ and $CC_1$ are concurrent at some point $F$.
\end{lemma}

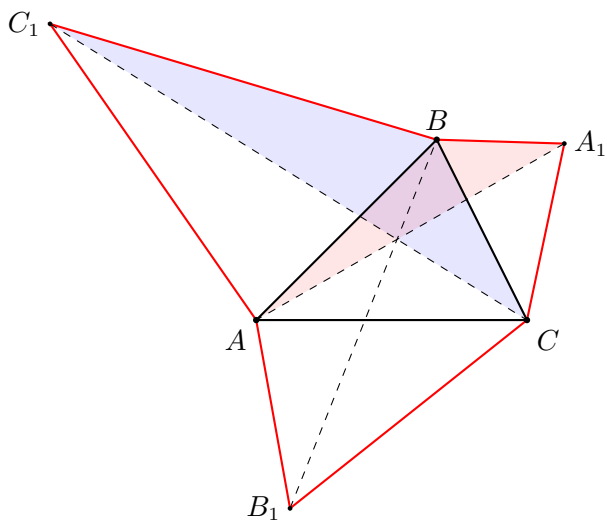
\begin{figure}[!hbt]
    \centering
    
\begin{tikzpicture}[scale=0.6]

\pgfmathsetmacro{\Ax}{0}   
\pgfmathsetmacro{\Ay}{0}
\pgfmathsetmacro{\Bx}{4}   
\pgfmathsetmacro{\By}{4}
\pgfmathsetmacro{\Cx}{6}   
\pgfmathsetmacro{\Cy}{0}

\pgfmathsetmacro{\La}{0.6}
\pgfmathsetmacro{\Lb}{0.4}
\pgfmathsetmacro{\Lc}{0.3}

\coordinate (A) at (\Ax,\Ay);
\coordinate (B) at (\Bx,\By);
\coordinate (C) at (\Cx,\Cy);

\pgfmathsetmacro{\ABlen}{veclen(\Bx-\Ax,\By-\Ay)}
\pgfmathsetmacro{\AClen}{veclen(\Cx-\Ax,\Cy-\Ay)}
\pgfmathsetmacro{\BClen}{veclen(\Cx-\Bx,\Cy-\By)}

\pgfmathsetmacro{\rBone}{\Lc/\La * \ABlen} 
\pgfmathsetmacro{\rCtwo}{\Lb/\La * \AClen} 

\pgfmathsetmacro{\dax}{\Cx-\Bx}
\pgfmathsetmacro{\day}{\Cy-\By}
\pgfmathsetmacro{\da}{veclen(\dax,\day)}

\pgfmathsetmacro{\a}{(\rBone*\rBone - \rCtwo*\rCtwo + \da*\da) / (2*\da)}
\pgfmathsetmacro{\hh}{\rBone*\rBone - \a*\a}
\pgfmathsetmacro{\pax}{\Bx + \a * \dax / \da}
\pgfmathsetmacro{\pay}{\By + \a * \day / \da}

\pgfmathsetmacro{\ixone}{\pax + sqrt(\hh) * (-\day) / \da}
\pgfmathsetmacro{\iyone}{\pay + sqrt(\hh) * (\dax)  / \da}
\pgfmathsetmacro{\ixtwo}{\pax - sqrt(\hh) * (-\day) / \da}
\pgfmathsetmacro{\iytwo}{\pay - sqrt(\hh) * (\dax)  / \da}

\pgfmathsetmacro{\sideA}{(\Ax-\Bx)*(\Cy-\By) - (\Ay-\By)*(\Cx-\Bx)}
\pgfmathsetmacro{\side1}{(\ixone-\Bx)*(\Cy-\By) - (\iyone-\By)*(\Cx-\Bx)}

\coordinate (Apint) at (\ixone,\iyone);


\pgfmathsetmacro{\ta}{((\ixone-\Bx)*\dax + (\iyone-\By)*\day) / (\da*\da)}
\pgfmathsetmacro{\xah}{\Bx + \ta * \dax}
\pgfmathsetmacro{\yah}{\By + \ta * \day}
\pgfmathsetmacro{\xar}{2*\xah - \ixone}
\pgfmathsetmacro{\yar}{2*\yah - \iyone}

\coordinate (Ap) at (\xar,\yar); 

\pgfmathsetmacro{\ratCA}{(\La*\BClen/(\Lc*\ABlen))}
\pgfmathsetmacro{\ratBA}{(\La*\BClen/(\Lb*\AClen))}


\pgfmathsetmacro{\angAB}{atan2(\Cy-\By,\Cx-\Bx)} 
\pgfmathsetmacro{\angBP}{atan2(\yar-\By,\xar-\Bx)}

\pgfmathsetmacro{\theta}{\angAB-\angBP}
\pgfmathsetmacro{\c}{cos(-\theta)}
\pgfmathsetmacro{\si}{sin(-\theta)}

\pgfmathsetmacro{\BX}{\Ax-\Bx} \pgfmathsetmacro{\BY}{\Ay-\By}
\pgfmathsetmacro{\Cpx}{\Bx + \ratCA*(\c*\BX - \si*\BY)}
\pgfmathsetmacro{\Cpy}{\By + \ratCA*(\si*\BX + \c*\BY)}
\coordinate (Cpint) at (\Cpx,\Cpy); 


\pgfmathsetmacro{\angAC}{atan2(\By-\Cy,\Bx-\Cx)} 
\pgfmathsetmacro{\angCP}{atan2(\yar-\Cy,\xar-\Cx)}

\pgfmathsetmacro{\thetac}{\angAC-\angCP}
\pgfmathsetmacro{\cC}{cos(-\thetac)}
\pgfmathsetmacro{\siC}{sin(-\thetac)}

\pgfmathsetmacro{\CX}{\Ax-\Cx} \pgfmathsetmacro{\CY}{\Ay-\Cy}
\pgfmathsetmacro{\Bpx}{\Cx + \ratBA*(\cC*\CX - \siC*\CY)}
\pgfmathsetmacro{\Bpy}{\Cy + \ratBA*(\siC*\CX + \cC*\CY)}
\coordinate (Bpint) at (\Bpx,\Bpy); 

\fill[opacity=0.1, blue] (C)--(B)--(Cpint);
\fill[opacity=0.1, red] (A)--(B)--(Apint);

\draw[thick] (A)--(B)--(C)--cycle;
\draw[thick,red] (B)--(Apint)--(C); 
\draw[thick,red] (B)--(Cpint)--(A);
\draw[thick,red] (A)--(Bpint)--(C);

\draw[dashed] (Apint) -- (A); 
\draw[dashed] (Cpint) -- (C);
\draw[dashed] (Bpint) -- (B);

\fill (A) circle(2pt) node[below left] {$A$};
\fill (B) circle(2pt) node[above] {$B$};
\fill (C) circle(2pt) node[below right] {$C$};
\fill (Apint) circle(1.5pt) node[right] {$A_1$}; 

\fill (Bpint) circle(1.5pt) node[left] {$B_1$}; 

\fill (Cpint) circle(1.5pt) node[left] {$C_1$}; 

\end{tikzpicture}

\caption{Main construction}
    \label{fig:placeholder}
\end{figure}

\begin{proof}
    Denote by $A_0$ the point of intersection of 
    $A_1A$ with $BC$ and similarly for $B_0$ and $C_0$. Then \begin{equation}\label{eq:mainCeva}\frac{\dis(C,A_0)}{\dis(B,A_0)}=\frac{[ACA_0]}{[ABA_0]}=\frac{[A_1CA_0]}{[A_1BA_0]}=\frac{[ACA_1]}{[ABA_1]}=\frac{\lambda_B b^2\sin(\gamma+\tilde{\gamma})}{\lambda_C c^2\sin(\beta+\tilde{\beta})}\end{equation} 
    where $\alpha$, $\beta$ $\gamma$ stand for the angles of the triangle $ ABC$ and $\tilde{\alpha}, \tilde{\beta}, \tilde{\gamma}$  denote the corresponding angles of the triangle $\tilde{\triangle}$.

    Hence, $\dfrac{\dis(C,A_0)}{\dis(B,A_0)}\dfrac{\dis(B,C_0)}{\dis(A,C_0)}\dfrac{\dis(A,B_0)}{\dis(C,B_0)}=1$ and the claim follows by Ceva theorem.
\end{proof}

\begin{remark}
    It worth mentioning that the above construction seems to be equivalent to the one, presented in \cite{Hess2015}. However, we haven't been able to match the expressions for barycentric coordinates, mentioned in this paper with our computations. 
\end{remark}

\begin{theorem}[Main Construction.]\label{th:main}
Point $F$, provided by Lemma \ref{lm:construction} is the Snell-Fagnano point $F_\lambda$.
\end{theorem}

\begin{proof}
   Triangles $ABA_1$ and $C_1BC$ are similar with the coefficient of proportionality $\frac{c\lambda_C}{a\lambda_A}$ and so
   $\angle AA_1B=\angle C_1CB.$
   
From the law of sines for the triangle $BAA_1$
 we get $\frac{\sin(BAA_1)}{\sin(BA_1A)}=\frac{\dis(B,A_1)}{\dis(B,A)}=\frac{\lambda_C}{\lambda_A}$ and the assertion of the theorem follows from the lemma \ref{lm:SFpoint}.     
\end{proof}

\begin{corollary} \label{cor:main}
   The Snell-Fagnano point $F_\lambda$ belongs to the interior of the triangle $ABC$ if
   \begin{equation}\label{eq:conditions}
   \begin{cases}
            \alpha+\tilde{\alpha}\leqslant \pi\\
            \beta+\tilde{\beta}\leqslant \pi\\
            \gamma+\tilde{\gamma}\leqslant \pi
        \end{cases}.\end{equation}
   \end{corollary}

\begin{proof}
    For the point $F$, provided by Lemma \ref{lm:construction} to reside in the interior of the triangle $ABC$ is equivalent to the quadrilaterals $ABA_1C$, $BCB_1A$ and $CAC_1B$ being convex.
\end{proof}

\begin{remark}
Strictly speaking, our construction provides only some of the points, satisfying the conditions of Lemma \ref{lm:SFpoint}. As it follows from Lemmas \ref{lm:isogonalFagnano} and \ref{lm:Akopyan} there are two such points, with at most one in the interior of the triangle. A more formal way to characterize whether the point $F$ belongs to the interior of the triangle $ABC$ is to ask for the positivity of all barycentric coordinates of $F$. Or, equivalently, ask for the positivity of all barycentric coordinates of the isogonal conjugate point $F^*$, having tripolar coordinates $(\lambda_A:\lambda_B:\lambda_C)$, thanks to Lemma \ref{lm:isogonalFagnano}.

Unfortunately, we haven't found correct formulas for the conversion from tripolar to barycentric coordinates in the existing literature (cf. \cite{Hess2015}, \cite{Casey}). Hence, we are providing our own expressions in the appendix and encourage the curious readers to convince themselves that the conditions \eqref{eq:conditions} are indeed equivalent to the positivity of the barycentric coordinates provided by the expressions  \eqref{eqn: tripolar}.

\end{remark}

\begin{remark}
    If on the other hand, none of the common points of the Apollonius circles belong to the interior of the triangle, arguments similar to the ones in the classical case show that the Snell-minimizer degenerates to the double cover of the shortest triangle altitude. We omit these case studies for the sake of brevity.  
\end{remark}
\section{Discussion and further directions.}\label{sec:other}
Next we will present few possible directions for the generalizations of our findings.

\subsection{Other polygons.}
The most natural generalization is to consider Snell billiards (or, equivalently, inscribed $n$-gons of minimal weighted perimeter) for $(n>3)$-gonal table (see e.g. \cite{Knox2023}, \cite{Davis_AdvGeom_2018}).

As there is no direct analog of the geometric methods used in this paper for general $n$-gons, we still believe that some natural questions might be addressed. For instance, it would be interesting to find the analog of lemma \ref{lm:isogonalFagnano}:
\emph{Is the existence of the minimal Snell's billiard trajectory with the weights $(\lambda_1,\dots,\lambda_n)$ equivalent to the existence of the point at the distances to the vertices proportional to $(\lambda_1:\lambda_2:\ldots:\lambda_n)$?}

\subsection{Other geometries.} As Fagnano problem can be posed in other geometries as well, it is natural to investigate weighted Fagnano problem for hyperbolic and spherical triangles. 

\subsection{Outer billiards.} One can also consider Snell modification of \emph{dual billiards}. Recall that the dual billiard map around the polygon is defined as a reflection in the right-most vertex of the polygon. One could consider the modified reflection with the scaling parameters $(\varkappa_1, \ldots,\varkappa_n)$, satisfying the conditions $\prod\limits_{j=1}^n \varkappa_j=1$ and look for the periodic trajectories of thus defined transformation. In the forthcoming paper \cite{Jay}
, the second author generalized the notion of the necklace dynamics from \cite{Gutkin} for the case of rational polygons and rational weights.

\subsection{Contractive and dispersive Snell billiards.} Lastly, one can consider the case of Snell's billiards where the conditions of lemma \ref{lm:point} are not satified. Thus if the product $\prod\limits_{j=1}^n \varkappa_j$ is less than one, we expect the behavior to be similar to the so-called slap map (see e.g. \cite{pinball},\cite{slap_map}, \cite{bolotin2016degenerate}).

\section{Acknowledgements}

 Authors are grateful to the MIT PRIMES-USA program, particularly Dr. Tanya Khovanova, Dr. Felix Gotti, Dr. Slava Gerovitch, and Prof. Pavel Etingof, for making this research possible.
 First author was supported by Simons foundation grant MPS-TSM-00013259

\bibliographystyle{plain}
\bibliography{billiards}

\appendix

\section{Conversion from tripolar to barycentric coordiantes.}

We will present a corrected version of the coordinate conversion formulas from \cite{Casey}. We will stick to the Conway notations (see e.g. \cite{Conway}).
We will denote by $\triangle$ the triangle with the sides $(a,b,c)$ and by $\tilde{\triangle}$ the triangle with the sides $(Xa,Yb,Zc)$.
   We let $S_a = \frac{b^2+c^2-a^2}{2}$,  $S_{\tilde{a}} = \frac{(Yb)^2+(Zc)^2-(Xa)^2}{2}$, and define $S_b, S_c$ , $S_{\tilde{b}}, and S_{\tilde{c}}$ similarly. 
    For the unsigned areas of the triangles we will use the notations $[\triangle]$ $[\tilde{\triangle}]$.

\begin{theorem}\label{thm: tripolar}
Let $\triangle= ABC$ be fixed, and let $P$ have tripolar coordinates $(X~:~Y~:~Z)$.  Denote 
$
F = (X^2 S_{\tilde{a}} + Y^2 S_{\tilde{b}} + Z^2 S_{\tilde{c}})
    - (a^2 Y^2 Z^2 + X^2 b^2 Z^2 + X^2 Y^2 c^2).
$
Then $P$ has normalized barycentric coordinates $(\rho_a,\rho_b,\rho_c)$ given by
\begin{equation}\label{eqn: tripolar}
    \begin{cases}
        \rho_a = \frac{1}{16[\triangle]}\left((S_c Y^2 + S_b Z^2 - 2 a^2 X^2)\,s^2 + a^2 S_a\right) \\
        \rho_b = \frac{1}{16 [\triangle]}\left((S_a Z^2 + S_c X^2 - 2 b^2 Y^2)\,s^2 + b^2 S_b\right) \\
        \rho_c = \frac{1}{16[\triangle]}\left((S_b X^2 + S_a Y^2 - 2 c^2 Z^2)\,s^2 + c^2 S_c\right) 
    \end{cases},
\end{equation}
where \footnote{ It is worth noticing that under our assumptions both choices of the right hand side give non-negative expressions.}
$
s^2 = \dfrac{(abc)^2 \pm 16\,[\triangle][\tilde{\triangle}]}{2F}.
$
\end{theorem}

\begin{proof}
Let $\dis(P,A)=X s$ and similarly for $\dis(P,B)$ and $\dis(P,C)$. For normalized barycentric coordinates $(\rho_a,\rho_b,\rho_c)$, the area of the triangle $PBC$ equals $\rho_a[\triangle]$ and similarly for other triangles. Then for the area of the triangle $ABC$ we get

 $$\begin{aligned}
     2[\triangle]&=bc\sin A=bc(\sin\angle BAP\cos\angle PAC+\cos \angle BAP\sin\angle PAC)=\\
     &=bc[\triangle]\left(\frac{2\rho_b}{b (Xs)}\frac{c^2+(Xs)^2-(Ys)^2}{2 c (Xs)}+\frac{2\rho_c}{c (Xs)}\frac{b^2+(Xs)^2-(Zs)^2}{2 b (Xs)}\right).
 \end{aligned}$$
So
 $$2(Xs)^2=(\rho_b(c^2+(Xs)^2-(Ys)^2)+\rho_c(b^2+(Xs)^2-(Zs)^2)$$
 and, since $\rho_a+\rho_b+\rho_c=1$,
 $$  (-Xs)^2 \rho_a+ (c^2-(Ys)^2)\rho_b+(b^2-(Zs)^2)\rho_c=(Xs)^2.$$

Similarly, we obtain two more linear equations from the areas of the triangles $PBC$ and $PCA$. Solving these three linear equations on $\rho_a,$ $\rho_b$ and $\rho_c$, we obtain the expressions \eqref{eqn: tripolar}.

    To solve for $t,$ we compute the intersections of the parametric line given by \eqref{eqn: tripolar} with the Apollonian circle on $\odot_{B,C}(Y/Z).$

    \begin{proposition}\label{prop: apollonian}
Let $P =(\rho_a, \rho_b, \rho_c)$ be \textit{normalized} barycentric coordinates, and let $Y, Z > 0$. Then the locus of points satisfying $\dis(B,P)/\dis(C,P) = Y/Z$ is 
\begin{equation}
\label{eq:apollonian}
(\rho_a^2c^2+\rho_c^2a^2+\rho_a\rho_c S_b)Z^2 = (\rho_a^2b^2 + \rho_b^2a^2+\rho_a\rho_b S_c)Y^2. 
\end{equation}
\end{proposition}
\begin{proof}
 Compute the lengths $\dis(B,P)$ and $\dis(C,P)$. We have 
    \begin{align*}
        \dis(B,P)^2 &= \norm{\rho_a(\vec{PA}-\vec{PB}) + \rho_c(\vec{PC}-\vec{PB})}^2 =
        \rho_a^2c^2 + \rho_c^2a^2 + \rho_a\rho_c S_b. 
    \end{align*}
    Similarly, 
    $
    \dis(C,P)^2 = \rho_a^2b^2 + \rho_b^2a^2+\rho_a\rho_b S_c, 
    $ and the result follows from the direct application of $
    \frac{\dis(B,P)^2}{\dis(CP)^2} = \frac{Y^2}{Z^2}.
    $
\end{proof}

Writing the barycentric coordinates in \eqref{eqn: tripolar} as
\begin{equation}
    \begin{cases}
        \rho_a = \alpha_0+ \alpha_1 s^2, \\
        \rho_b = \beta_0 + \beta_1 s^2, \\
        \rho_c = \gamma_0 + \gamma_1 s^2,
    \end{cases}
\end{equation}
where $\alpha_0 = a^2S_a, \alpha_1 = S_cY^2 + S_bZ^2 - 2a^2X^2$, and so on. Substituting these values into the equation \eqref{eq:apollonian} and rearranging gives the biquadratic equation
\[
A_2(s^2)^2+A_1s^2+A_0 = 0, 
\] where, letting $k=Y/Z$, we have
\begin{equation}
    \begin{cases}
        A_{2} &= (c^{2} - k^{2} b^{2}) \alpha_{1}^{2}
+ a^{2}(\gamma_{1}^{2} - k^{2}\beta_{1}^{2})
+ S_{B}\alpha_{1}\gamma_{1}
- k^{2} S_{C}\alpha_{1}\beta_{1}, 
\\A_{1} &= 2(c^{2} - k^{2} b^{2}) \alpha_{0}\alpha_{1}
+ 2a^{2}(\gamma_{0}\gamma_{1} - k^{2}\beta_{0}\beta_{1})
\\&\qquad+ S_{B}(\alpha_{0}\gamma_{1} + \alpha_{1}\gamma_{0})
- k^{2} S_{C}(\alpha_{0}\beta_{1} + \alpha_{1}\beta_{0}), \\
A_{0} &= (c^{2} - k^{2} b^{2}) \alpha_{0}^{2}
+ a^{2}(\gamma_{0}^{2} - k^{2}\beta_{0}^{2})
+ S_{B}\alpha_{0}\gamma_{0}
- k^{2} S_{C}\alpha_{0}\beta_{0}.
    \end{cases}
\end{equation}
 The solution to this biquadratic is 
the one presented in the statement. Since $F$ is positive and $a^2b^2c^2\geqslant 16[\triangle][\tilde{\triangle}]$ both roots are positive and thus the theorem \ref{thm: tripolar} is proven.
\end{proof}

\end{document}